\documentclass[12pt]{amsart}

\usepackage{amsfonts, amsmath, amscd}
\usepackage[psamsfonts]{amssymb}

\usepackage{amssymb}

\usepackage{pb-diagram}

\newtheorem{theorem}{Theorem}[section]
\newtheorem{lemma}[theorem]{Lemma}

\newtheorem{question}[theorem]{Question}

\numberwithin{equation}{section}

\newcommand{\DD}{\mathcal{D}}
\newcommand{\GG}{\mathfrak{G}}

\begin{document}

\title[Free locally convex spaces having countable tightness]{A characterization of  free locally convex spaces over metrizable spaces which have countable tightness}

\author{S. S. Gabriyelyan}
\address{Department of Mathematics, Ben-Gurion University of the Negev, Beer-Sheva P.O. 653, Israel}
\email{saak@math.bgu.ac.il}



\begin{abstract}
We prove that the free locally convex space $L(X)$ over a metrizable space $X$ has countable tightness if and only if $X$ is separable.
\end{abstract}

\maketitle

\section{Introduction}

A topological space $X$ is called {\it first countable} if it has a countable open base at each point. Any first countable topological group is metrizable. Various topological properties generalizing first countability have been studied intensively by topologists and analysts, especially Fr\'{e}chet-Urysohness, sequentiality, to be a $k$-space and countable tightness (see \cite{Eng,kak}). It is well know that, 
metrizability $\Rightarrow$ Fr\'{e}chet-Urysohness $\Rightarrow$ sequentiality $\Rightarrow$ countable tightness, and sequentiality $\Rightarrow$ to be a $k$-space. Although none of these implications is reversible, for many important classes of locally convex spaces (lcs for short) some of them can be reversed. K{\c{a}}kol 
showed that for an $(LM)$-space (the inductive limit of a sequence of locally convex metrizable spaces), metrizability $\Leftrightarrow$ Fr\'{e}chet-Urysohness. The Cascales and Orihuela result states that for an $(LM)$-space, sequentiality $\Leftrightarrow$ to be a $k$-space.
Moreover, K{\c{a}}kol  and Saxon \cite{KS} proved the next structure theorem: An $(LM)$-space $E$ is sequential (or a $k$-space) if and only if $E$ is metrizable or is a Montel $(DF)$-space.
Topological properties of a lcs  $E$ in the weak topology $\sigma(E,E')$ are of the importance and have been intensively studied from many years (see  \cite{kak,bonet}).   Corson (1961)  started a systematic study of certain topological properties of the weak topology of Banach spaces. If $B$ is any infinite-dimensional Banach space, a classical result of Kaplansky states that $(E,\sigma(E,E'))$ has countable tightness (see \cite{kak}), but the weak dual $(E',\sigma(E',E))$ is not a $k$-space (see \cite{KS}). Note that there exists a $(DF)$-space with uncountable tightness whose weak topology has countable tightness \cite{CKS}.
 We refer the reader to the book \cite{kak} for many references and facts.



In this paper we consider another class in the category $\mathbf{LCS}$ of locally convex spaces and continuous linear operators which is the most important from the categorical point of view, namely the class of free locally convex spaces over Tychonoff spaces introduced by Markov \cite{Mar}.  Recall that the {\it  free locally convex space} $L(X)$ over a Tychonoff space $X$ is a pair consisting of a locally convex space $L(X)$ and  a continuous mapping $i: X\to L(X)$ such that every  continuous mapping $f$ from $X$ to a locally convex space $E$ gives rise to a unique continuous linear operator ${\bar f}: L(X) \to E$  with $f={\bar f} \circ i$. The free locally convex space $L(X)$  always exists and is  unique. The set $X$ forms a Hamel basis for $L(X)$, and  the mapping $i$ is a topological embedding \cite{Rai, Flo1, Flo2, Usp}. It turns out that excepting the trivial case when $X$ is a countable discrete space, the free  lcs $L(X)$ is never a $k$-space \cite{Gabr}: For  a Tychonoff space $X$,  $L(X)$ is a $k$-space if and only if $X$ is a countable discrete space.

The aforementioned results explain our interest to the following problem.
\begin{question} \label{qFreeTight}
For which Tychonoff spaces $X$ the free  lcs $L(X)$ has countable tightness?
\end{question}

We obtain a complete answer to Question \ref{qFreeTight} for the important case when $X$ is metrizable. The following theorem is the main result of the article.
\begin{theorem} \label{tFreeMain}
Let $X$ be a metrizable space. Then the free lcs $L(X)$ has countable tightness if and only if $X$ is separable.
\end{theorem}
Below we prove even a stronger result (see Theorem \ref{tFreeV}).

\section{Proof of Theorem \ref{tFreeMain}}

The free (resp. abelian) topological group $F(X)$ (resp. $A(X)$) over a Tychonoff space $X$ were also introduced by Markov \cite{Mar} and intensively studied over the last half-century (see \cite{Gra,MMO,Rai,Tkac,Usp}), we refer the reader to \cite[Chapter 7]{ArT} for basic definitions and results.
We note that the topological  groups $F(X)$ and $A(X)$ always exist and are essentially unique. 
Note also that the identity map $id_X :X\to X$ extends to a canonical homomorphism $id_{A(X)}: A(X)\to L(X)$ which is an embedding of topological groups \cite{Tkac, Usp2}.

Recall that a space $X$ has {\it countable tightness} if whenever $x\in \overline{A}$ and $A\subseteq X$, then $x\in \overline{B}$ for some countable $B\subseteq A$.
We use the following remarkable result of Arhangel'skii, Okunev and Pestov which shows that the topologies of $F(X)$ and $A(X)$ are rather complicated and unpleasant even for the simplest case of a metrizable space $X$.
\begin{theorem}[\cite{AOP}] \label{tAOP}
Let $X$ be a metrizable space. Then:
\begin{enumerate}
\item[{\rm (i)}] The tightness of $F(X)$ is countable if and only if $X$ is separable or discrete.
\item[{\rm (ii)}] The tightness of $A(X)$ is countable if and only if the set $X'$ of all non-isolated points in $X$  is separable.
\end{enumerate}
\end{theorem}

For the case $X$ is  discrete (hence metrizable) we have the following.
\begin{theorem}[\cite{Gabr}] \label{tFreeSpace}
For each uncountable discrete space $D$, the space $L(D)$ has uncountable tightness.
\end{theorem}

The space of all continuous functions on a topological space $X$ endowed with the compact-open topology we denote by $C_c(X)$. It is well known that the space $L(X)$ admits a canonical continuous monomorphism $L(X)\to C_c (C_c (X))$. If $X$ is a $k$-space, this monomorphism is an embedding of lcs \cite{Flo1, Flo2, Usp}. So, for  $k$-spaces, we obtain the next chain of topological embeddings:
\begin{equation} \label{emb}
A(X) \hookrightarrow L(X) \hookrightarrow C_c (C_c (X)).
\end{equation}


Pytkeev \cite{Pyt} proved that every sequential space satisfies the property which is stronger than countable tightness. Following \cite{malyhin}, we say that a topological space $X$ has the {\it Pytkeev property at a point $x\in X$} if for each $A\subseteq X$ with $x\in \overline{A}\setminus A$, there are infinite subsets $A_1, A_2, \dots $ of $A$ such that each neighborhood of $x$ contains some $A_n$. In \cite{boaz} this property is strengthened as follows. A  topological space $X$ has  the {\it strong Pytkeev property at a point $x\in X$} if there exists a countable family $\DD$  of subsets of $X$, which is called a {\it Pytkeev network at $x$},   such that for each neighborhood $U$ of $x$ and  each $A\subseteq X$ with $x\in \overline{A}\setminus A$, there is $D\in\DD$ such that $x\in D\subseteq U$ and $D\cap A$ is infinite. Following \cite{Banakh}, a space $X$ is called a {\it Pytkeev $\aleph_0$-space} if $X$ is regular and has  a countable family $\DD$ which is a Pytkeev network at each point $x\in X$. The strong Pytkeev property for topological groups is thoroughly studied in \cite{GKL2}, where, among others, it is proved that $A(X)$ and $L(X)$ have the strong Pytkeev property for each $\mathcal{MK}_\omega$-space $X$ (i.e., $X$ is the inductive limit of an increasing sequence of compact metrizable subspaces). Note also that in general (see \cite{GKL2}):  Fr\'{e}chet-Urysohness $\not\Rightarrow$  the strong   Pytkeev property  $\not\Rightarrow$  $k$-space.

Recall that a family $\mathcal{D}$ of subsets of a topological space $X$ is called a {\it $k$-network} in $X$ if, for every compact subset $K\subset X$ and each neighborhood $U$ of $K$ there exists a finite subfamily $\mathcal{F}\subset \mathcal{D}$ such that $K\subset \bigcup \mathcal{F}\subset U$. Following Michael \cite{Mich}, a topological space $X$ is called an {\it $\aleph_0$-space} if it is  regular and has a countable $k$-network. Every separable and metrizable space is a Pytkeev $\aleph_0$-space, and every Pytkeev $\aleph_0$-space is an $\aleph_0$-space \cite{Banakh}. We use that following strengthening of Michael's theorem \cite{Mich}:
\begin{theorem}[\cite{Banakh}] \label{tBanakh}
If $X$ is an $\aleph_0$-space, then $C_c(X)$ is a Pytkeev $\aleph_0$-space.
\end{theorem}


The next theorem is an easy corollary of (\ref{emb}) and Theorem \ref{tBanakh}, the implication (iii)$\Rightarrow$(ii) was first observed by A. Leiderman (see \cite{Banakh}).
\begin{theorem} \label{tAleph}
For a $k$-space $X$ the following assertions are equivalent:
\begin{enumerate}
\item[{\rm (i)}]  $A(X)$  is a Pytkeev $\aleph_0$-space.
\item[{\rm (ii)}]  $L(X)$  is a Pytkeev $\aleph_0$-space.
\item[{\rm (iii)}]  $X$  is a Pytkeev $\aleph_0$-space.
\end{enumerate}
\end{theorem}

\begin{proof}
The implications (ii)$\Rightarrow$(i) and (i)$\Rightarrow$(iii) immediately follow from the fact that $A(X)$ is a subspace of $L(X)$ and $X$ is a subspace of $A(X)$.

(iii)$\Rightarrow$(ii). Assume that $X$ is a Pytkeev $\aleph_0$-space. Then $C_c(X)$ and  $C_c (C_c (X))$ are  Pytkeev $\aleph_0$-space by Theorem \ref{tBanakh}.  Since $X$ is a $k$-space,  $L(X)$ is a subspace of $C_c (C_c (X))$ by (\ref{emb}). Thus $L(X)$ is a Pytkeev $\aleph_0$-space.
\end{proof}

We need the next lemma.
\begin{lemma} \label{lFG}
If $U$ is a clopen subset of a Tychonoff space $X$, then $L(U)$ embeds into $L(X)$ as a closed subspace.
\end{lemma}

\begin{proof}
Denote by $\tau_U$ the topology of $L(U)$ and by $L(U)_a$ the underlying free vector space generated by $U$. Fix a point $e$ belonging to $U$. Let $i: U\to X$ be the natural inclusion. By the definition of $L(U)$, $i$ can be extended to a continuous inclusion $\widetilde{i}: L(U)\to L(X)$. So $\tau_U$ is stronger than the topology $\tau^X_U$ induced on $L(U)_a$ from $L(X)$. Define now $p: X\to U$ as follows: $p(x)=x$ if $x\in U$, and $p(x)=e$ if $x\in X\setminus U$. Clearly, $p$ is continuous. By the definition of $L(X)$, $p$ can be extended to a continuous linear mapping $\widetilde{p}: L(X)\to L(U)$. Since $p\circ  i=\mathrm{id}_U$, we obtain $\widetilde{p}\circ \widetilde{i}=\mathrm{id}_{L(U)}$ and $\widetilde{p}$ is injective on $L(U)_a$. So $\tau^X_U$ is stronger than the topology $\tau_U$. Thus $\tau^X_U =\tau_U$.
Since $U$ is a closed subset of $X$ the subspace $L(U,X)$ of $L(X)$ generated by $U$ is closed
(we can repeat word for word the proof of Proposition 3.8 in \cite{STh}). Thus $\widetilde{i}$ is an embedding of $L(U)$ onto the closed subspace $L(U,X)$ of $L(X)$.
\end{proof}

Now Theorem \ref{tFreeMain} is a part of the following theorem.

\begin{theorem} \label{tFreeV}
For a metrizable space $X$ the following assertions are equivalent:
\begin{enumerate}
\item[{\rm (i)}]  $L(X)$ is a Pytkeev $\aleph_0$-space.
\item[{\rm (ii)}]  $L(X)$ has countable tightness.
\item[{\rm (iii)}]  $X$ is separable.
\end{enumerate}
\end{theorem}

\begin{proof}
(i)$\Rightarrow$(ii) is clear. Let us prove 
(ii)$\Rightarrow$(iii). Since $A(X)$ is a subgroup of $L(X)$, we obtain that $A(X)$ also has countable tightness. Now Theorem \ref{tAOP} implies that the set $X'$ of all non-isolated points of $X$ is separable. So we have to show only that the set $D$ of all isolated points of $X$ is countable.

Suppose for a contradiction that $D$ is uncountable. Then there is a positive number $c$ and an uncountable subset $D_0$ of $D$ such that $B_c(d)=\{ d\}$ for every $d\in D_0$, where $B_c(d)$ is the $c$-ball centered at $d$. It is easy to see that $D_0$ is a clopen  subset of $X$. So, by Lemma \ref{lFG}, $L(D_0)$ is a subspace of $L(X)$. Now Theorem \ref{tFreeSpace} yields that $L(D_0)$ and hence $L(X)$ have uncountable tightness. This contradiction shows that $D$ is countable. Thus $X$ is separable.

(iii)$\Rightarrow$(i) immediately follows from Theorem \ref{tAleph}.
\end{proof}

We do not know whether the assertions (i) and (ii) in Theorem \ref{tAOP} are equivalent respectively  to the following: $F(X)$ is a Pytkeev $\aleph_0$-space, and $A(X)$ has the strong Pytkeev property.


\bibliographystyle{amsplain}

\end{document}